\documentclass[11pt]{llncs}
\usepackage[utf8]{inputenc}
\usepackage{amsmath,amssymb}
\usepackage{hyperref}
\hypersetup{
    colorlinks=true,
    linkcolor=blue,
    filecolor=blue,      
    urlcolor=blue}

\usepackage{comment,pifont,graphicx,tikz,soul,xcolor,longtable,todonotes,listings,algorithm,diagbox}
\usepackage{tikz-cd}
\usepackage[margin=3cm]{geometry}
\usetikzlibrary{decorations.pathreplacing}

\newcommand{\Ell}{\mathcal{E}\hspace{-0.065cm}\ell\hspace{-0.035cm}\ell}
\newcommand{\F}{\mathbb{F}}

\newcommand{\C}{\mathbb{C}}

\newcommand{\Z}{\mathbb{Z}}

\newcommand{\Q}{\mathbb{Q}}

\DeclareMathOperator{\cl}{cl}
\DeclareMathOperator{\id}{id}
\DeclareMathOperator{\GL}{GL}
\DeclareMathOperator{\End}{End}

\DeclareMathOperator{\charac}{char}

\newcommand{\afrak}{\mathfrak{a}}

\newcommand{\mfrak}{{\mathfrak{m}}}
\newcommand{\hfrak}{{\mathfrak{h}}}

\title{Generalized class group actions on oriented elliptic curves with level structure}
\date{}

\begin{document}
\author{%
Sarah Arpin\inst{1}  \and 
Wouter Castryck\inst{2} \and 
Jonathan Komada Eriksen\inst{3} \and
Gioella Lorenzon\inst{2} \and
Frederik Vercauteren\inst{2} 
}%
\institute{
Mathematics Institute,
Universiteit Leiden,
Leiden, The Netherlands and the Quantum Software Consortium of the Netherlands \\
\and
COSIC, ESAT, KU Leuven, Belgium\\
\and
Department of Information Security and Communication Technology, Norwegian University of Science and Technology, Trondheim, Norway}

\maketitle

\vspace{-1ex}
\begingroup
  \makeatletter
  \def\@thefnmark{$\ast$}\relax
  \@footnotetext{\relax
This work was supported in part by the European Research Council (ERC) under the European Union’s Horizon 2020 research and innovation programme (grant agreement ISOCRYPT - No. 101020788) and by CyberSecurity Research Flanders
with reference number VR20192203.}
\endgroup
\vspace{-1ex}

\begin{abstract}
    We study a large family of generalized class groups of imaginary quadratic orders $O$ and prove that they act freely and (essentially) transitively on the set of primitively $O$-oriented elliptic curves over a field $k$ (assuming this set is non-empty) equipped with appropriate level structure. This extends, in several ways, a recent observation due to Galbraith, Perrin and Voloch for the ray class group. We show that this leads to a reinterpretation of the action of the class group of a suborder $O' \subseteq O$ on the set of $O'$-oriented elliptic curves, discuss several other examples, and briefly comment on the hardness of the corresponding vectorization problems.
\end{abstract}

\section{Introduction}

A current trend in isogeny-based cryptography is to study elliptic curve isogenies that respect certain \textit{level structure}, or additional information about the curves. 
This has two main catalysts. Firstly, the recently established rapid mixing properties of isogeny graphs of supersingular elliptic curves with a marked cyclic subgroup of order $N$ ($\Gamma_N^0$-level structure) have led to improved security foundations, e.g., of the distributed generation of supersingular elliptic curves with unknown endomorphism ring~\cite{arpin,secuer}; see~\cite{codogni} for a generalization of these rapid mixing results to arbitrary level structure. Secondly, Robert's unconditional break of SIDH~\cite{RobSIDH} has revealed that the problem of finding an isogeny between two elliptic curves with full $\Gamma_N$-level structure is dramatically easier than in the case of plain elliptic curves, at least for $N$ smooth and large enough compared to the degree of the isogeny. The security of several recently proposed variants of SIDH~\cite{msidh,festa} also reduces to leveled isogeny problems. Some of these can again be broken much more efficiently than in the unleveled case; see~\cite{defeomodular} for a recent, systematic discussion.


In this paper we find explicit generalized class groups which act on the set of isomorphism classes of elliptic curves with various types of level structure. In doing so, we connect the study of isogenies between elliptic curves with level structure to the study of class group actions in the oriented framework of Col\`o--Kohel~\cite{OSIDH} and Onuki~\cite{onuki}. Briefly recall that, for $K$ an imaginary quadratic field, a $K$-orientation on an elliptic curve $E$ over a field $k$, say of positive characteristic $p$, is an embedding $\iota : K \hookrightarrow \End^0(E)$
into the endomorphism algebra of $E$ (assuming that such an embedding exists). For an order $O$ of $K$, such an orientation is a primitive $O$-orientation if $O = \iota^{-1}\End(E)$. 
For a fixed order $O \subseteq K$, the set $\Ell_k(O)$ of primitively $O$-oriented elliptic curves up to isomorphism naturally comes equipped with a free and (essentially) transitive action of the class group $\cl_O$ by  isogenies, see Section~\ref{ssec:CMaction} for more details. Isogenies arising from this class group action are called horizontal. For suitable parameters, this is considered a cryptographic group action, underpinning  constructions like CRS~\cite{couveignes,RostStol}, CSIDH~\cite{csidh,csurf} and SCALLOP~\cite{scallop,scallopHD}.

Level structures have sneaked up in the oriented setting before, although the situation is more diffuse than in the non-oriented case. E.g., for $N$ any prime different from $p$ that splits in $O$, it is well-known that horizontal $N$-isogenies automatically preserve the 
two eigenspaces of any generator $\sigma$ of $O$ acting on the $N$-torsion~\cite{msidh_artificial,defeomodular}; this can be seen as level structure 
amounting to the specification of two independent subgroups of order $N$. In~\cite{chenusmith}, Chenu and Smith study ideal class groups acting on supersingular elliptic curves $E / \F_{p^2}$ together with an $N$-isogeny to their Frobenius conjugate; this can be viewed as $\Gamma^0_N$-level structure, see also~\cite{xiao2023oriented}.
However, our starting point is a recent observation due to Galbraith, Perrin and Voloch~\cite{GPV} in the case of supersingular elliptic curves over $\F_p$, which is the setting of CSIDH: Just as $\cl_{\mathbb{Z}[\sqrt{-p}]}$ acts on the set of supersingular elliptic curves over $\F_p$ with $\F_p$-rational endomorphism ring $\Z[\sqrt{-p}]$, its ray class group for modulus $(N)$ acts on supersingular elliptic curves over $\F_p$ with full $\Gamma_N$-level structure.
In \cite{GPV}, the authors observe that although using supersingular elliptic curves over $\F_p$ equipped with such level structure yields a larger key space for the usual CSIDH parameters,  the security of such an enhanced protocol immediately reduces to that of the original CSIDH protocol, so this does not provide an advantage. We generalize this observation and contrast it with class group actions of suborders as in \cite{scallop}.

Our goal is to analyze to what extent the observation by Galbraith et al.\ is part of a bigger story. 
Instead of starting from a type of level structure and trying to devise a corresponding class group action, we invert the viewpoint and start from the action of a \emph{generalized class group}, subsequently finding the correct level structure in order to have a group action. In Section~\ref{sec:background} we provide an overview of the relevant background on elliptic curves with level structure, orientations, and generalized class groups. 
Our main results are discussed in Section~\ref{sec:gen_CM_action}, where we study a large family of generalized class groups, study their properties, and show that they act freely and (essentially) transitively on oriented elliptic curves with suitable level structure. We discuss several interesting examples, one of which sheds a new light on actions by class groups of non-maximal orders.
Finally, in Section~\ref{sec:security}, we discuss the hardness of the vectorization problem for our generalized class group actions. 

\subsection*{Acknowledgments} We thank the anonymous reviewers of the International Workshop on the Arithmetic of Finite Fields 2024 for several helpful comments. When preparing the final version of this paper, we were informed that Derek Perrin has independently been working on related ideas~\cite{perrin}. 

\section{Background}\label{sec:background}

\subsection{Elliptic Curves with Level Structure}
In this section, fix an integer $N\geq 2$ and a field $k$ such that $\charac k\nmid N$, and let $E$ be an elliptic curve over $k$. Let $E[N] \cong \mathbb{Z}/N\mathbb{Z}\times \mathbb{Z}/N\mathbb{Z}$ denote the $N$-torsion group of $E$. 
\begin{definition}[Level structure]\label{def:level-structure-classic}
    Let $\Gamma$ be a subgroup of $\GL_2(\mathbb{Z}/N\mathbb{Z})$. A $\Gamma$-level structure on an elliptic curve $E$ is a choice of isomorphism $\Phi:\mathbb{Z}/N\mathbb{Z}\times \mathbb{Z}/N\mathbb{Z} \cong E[N]$, up to pre-composition with an element of $\Gamma$. We denote this by the triple $(E,\Phi,\Gamma)$, just writing $(E,\Phi)$ when $\Gamma$ is understood from context.
\end{definition}

Choosing such an isomorphism $\Phi: \mathbb{Z}/N\mathbb{Z}\times \mathbb{Z}/N\mathbb{Z} \cong E[N]$ amounts to specifying a basis $P, Q \in E[N]$ and considering it up to base change by matrices from the prescribed group $\Gamma \subseteq \GL_2(\Z/(N))$. 
An isogeny $\varphi : E_1 \to E_2$ respects the level structures $P_1, Q_1$ resp.\ $P_2, Q_2$ if and only if $\varphi(P_1) = Q_1$, $\varphi(P_2)=Q_2$, modulo the action of $\Gamma$. 
Commonly studied examples are
\[ \Gamma_N^0 = \left\{ \begin{pmatrix} \ast & \ast \\ 0 & \ast \end{pmatrix} \right\},\quad \Gamma_N^{0,0} = \left\{ \begin{pmatrix} \ast & 0 \\ 0 & \ast \end{pmatrix} \right\}, \quad \Gamma_N^1 = \left\{ \begin{pmatrix} 1 & \ast \\ 0 & \ast \end{pmatrix} \right\} , \quad \Gamma_N = \left\{ \begin{pmatrix} 1 & 0 \\ 0 & 1 \end{pmatrix} \right\}  \]
where the level structure corresponds to specifying a cyclic subgroup of order $N$, two independent cyclic subgroups of order $N$, a point of order $N$, or a basis of $E[N]$ (``full level structure"), respectively.

\begin{example}\label{ex:level_structure}
    Fix an isomorphism $\Phi:(\mathbb{Z}/N\mathbb{Z}\times \mathbb{Z}/N\mathbb{Z})\to E[N]$ 
and let $P = \Phi(1,0)$, $Q = \Phi(0,1)$.
For any 
\[ h = \begin{pmatrix}
    a & b \\ 0 & d
\end{pmatrix}\in \Gamma^0_N, \] 
$\Phi\circ h$ sends $(1,0)\mapsto aP$ and $(0,1)\mapsto bP + dQ$. The basis element $(0,1)$ can be sent anywhere in the group $E[N]$ via $\Gamma_N^0$, but the image of $(1,0)$ is always in $\langle P\rangle$. In this sense $\Phi$ fixes a choice of cyclic subgroup $\langle P \rangle$. 
\end{example}

\subsection{Congruence subgroups and generalized class groups}

Our main references for this and the next section are~\cite{cox,kopplagarias2022}. For $O$ an order in an imaginary quadratic number field $K$, 
a \emph{modulus in $O$} is a non-zero integral ideal $\mfrak \subseteq O$. We denote by  $I_O$ the group of proper fractional ideals in $K$, which we recall are lattices $\afrak \subseteq K$ such that
$ \{ \, \alpha \in K \, | \, \alpha \afrak \subseteq \afrak \, \} = O$. 
Let $P_O$ be the subgroup of principal fractional ideals, i.e., fractional ideals of the form $\alpha O$ with $\alpha \in K \setminus \{ 0 \}$. Then the class group of $O$ is the quotient $\cl_O = I_O/P_O$. 

It can be shown that for any choice of modulus $\mfrak$, every class in $\cl_O$ contains an ideal $\afrak \subseteq O$ that is coprime with $\mfrak$, i.e., $\afrak + \mfrak = O$~\cite[Cor.\,7.17]{cox}. Equivalently, if we define
$I_O(\mfrak) \subseteq I_O$ to be the subgroup generated by all proper integral ideals that are coprime with $\mfrak$, 
and we let $P_O(\mfrak) = P_O \cap I_O(\mfrak)$,\footnote{Equivalently, $P_O(\mfrak)$ is the subgroup of $P_O$ generated by all principal integral ideals of $O$ that are coprime with $\mfrak$; see~\cite[\S4.3]{kopplagarias2022}, or see~\cite[Pf.~of~Prop.\,7.19]{cox} for the case $\mfrak = fO$ with $f \in \Z_{>0}$.} then the natural map $I_O(\mfrak)/P_O(\mfrak) \to \cl_O : [\afrak] \mapsto [\afrak]$ is an isomorphism.

A \emph{ray} for modulus $\mfrak$ is a principal fractional ideal of the form
\[ \alpha O \text{ with $\alpha \in K^\ast$ such that $\alpha \equiv 1 \bmod \mfrak$}, \]
where a congruence $\alpha \equiv \beta \bmod \mfrak$ means that for any $\alpha_1, \alpha_2, \beta_1, \beta_2 \in O$ such that $\alpha = \alpha_1/\alpha_2$ and $\beta = \beta_1 / \beta_2$ we have $\alpha_1 \beta_2 - \alpha_2 \beta_1 \in \mfrak$; see~\cite[Def.~4.2]{kopplagarias2022}.

The rays form a subgroup $P_{O, 1}(\mfrak) \subseteq P_O(\mfrak)$ called the \emph{ray group} for modulus $\mfrak$. Any group $H$ such that
$P_{O, 1}(\mfrak) \subseteq H \subseteq I_O(\mfrak)$
is then called a \emph{congruence subgroup} for modulus $\mfrak$. 
The corresponding quotient $I_O(\mfrak)/H$ is known as a \emph{generalized class group}; such groups play a crucial role in the study of abelian extensions of $K$. In the extremal case $H = P_{O, 1}(\mfrak)$
one ends up with the ray class group $\cl_{O, 1}(\mfrak)$. It was observed by Galbraith et al.\ that $\cl_{\Z[\sqrt{-p}],1}(N \Z[\sqrt{-p}])$ acts freely on the set of supersingular elliptic curves with $\F_p$-rational endomorphism ring $\Z[\sqrt{-p}]$ equipped with full $N$-level structure~\cite[Prop.\,2.5]{GPV}, where $N$ denotes an integer coprime to $p$; we will generalize this in
Section~\ref{sec:gen_CM_action}, but first we discuss how class groups of non-maximal orders are related to generalized class groups of their superorders.


\subsection{Class groups of suborders}


Let $O' \subseteq O$ be orders in $K$. Then $O' = \Z + fO$ for a unique positive integer $f$ called the \emph{conductor} of $O'$ relative to $O$. The following result shows that $\cl_{O'}$ can also be viewed as a generalized class group of $O$:


\begin{theorem} \label{thm:clOisgeneralized}
Let
\[ P_{O, \Z}(fO) = \{ \, \alpha O \, | \, \alpha \in K^\ast \text{ and } \alpha \equiv g \bmod fO \text{ for some $g \in \Z$ coprime with $f$} \, \}. \]
Then the map
\[ \cl_{O'} \to I_O(fO)/P_{O, \mathbb{Z}}(f O) : [\afrak] \mapsto [\afrak O], \]
where it can be assumed that $[\afrak]$ is represented by an integral $O'$-ideal that is coprime with $f O'$,
is an isomorphism of groups.
\end{theorem}
\begin{proof}
    This is the relative version of~\cite[Prop.\,7.22]{cox}, with the same proof. \qed
\end{proof}



The following classical exact sequence provides a foundational framework for working with generalized class groups of orders.


\begin{theorem}[{\cite[Theorem 5.4]{kopplagarias2022}}]\label{thm:kopplagarius}
    Let $K$ be an imaginary quadratic number field with ring of integers $O_K$ and orders $O'\subseteq O\subseteq O_K$. Let $f$ be the conductor of $O'$ relative to $O$. Then, the following sequence is exact:
        \begin{center}
    \begin{tikzcd}
        1\arrow[r] & O^\times/(O')^\times \arrow[r] & (O/fO)^\times/(O'/fO)^\times \arrow[r] & \cl_{O'}\arrow[r] & \cl_O\arrow[r] & 1.
    \end{tikzcd}
    \end{center}
\end{theorem}
\begin{proof}
    This is a generalization of the classical exact sequences with $O = O_K$ {\cite[Theorem I.12.12]{neukirch}}. The proof follows from \cite[Thm.\,5.4]{kopplagarias2022} by choosing $\mathfrak{m} = O,\mathfrak{m}' = O',\mathfrak{d} = fO'$ (notice that the roles of $O$ and $O'$ are reversed). \qed
\end{proof}

\subsection{Class group actions on sets of elliptic curves} \label{ssec:CMaction}


In this section, we assume to be working in characteristic $p > 0$ and let $k \subset \overline{\F_p}$. We describe the class group actions on certain sets of elliptic curves over finite fields. We begin following an approach of Waterhouse \cite{Wat69}, who provides such a group action for isomorphism classes of ordinary elliptic curves over finite fields. The endomorphism ring of an ordinary elliptic curve is isomorphic to an imaginary quadratic order $O$. 
The class group of the endomorphism ring is used to define a group action on the set of isomorphism classes of curves with isomorphic endomorphism rings. We write $\End(E)$ to denote the ring of all endomorphisms of $E/k$, defined over $\overline{k} = \overline{\F_p}$. When it arises, we write $\End_k(E)$ to specify the subring of endomorphisms of $E$ defined over $k$.

\begin{theorem}[{\cite[Theorem 4.5]{Wat69}}]\label{thm:waterhouse}
Let $E$ be an ordinary elliptic curve over a field $k \subset \overline{\F_p}$ having endomorphism ring $\End(E) \cong O$. Then the class group of $O$ acts freely and transitively on the set of elliptic curves over $k$ with endomorphism ring isomorphic to $O$. 
\end{theorem}

The ideal class group of $\End(E)$ acts on the set of isomorphism classes of ordinary elliptic curves with endomorphism rings isomorphic to the order $O$ in the following sense:

\begin{definition}\label{def:ideal_action}
    Let $E/k$ be an elliptic curve over a field $k \subset \overline{\F_p}$ with commutative endomorphism ring $\End(E)$. Take a proper integral ideal $I$ of $\End(E)$ with $N(I)$ coprime to $p$. Define 
    \[E[I]:= \bigcap_{\alpha\in I}\ker\alpha.\]
    As $I$ is a finitely generated $\Z$-module, the set $E[I]$ is a finite group. This finite group defines an isogeny $\varphi_I:E\to E/E[I]$ with kernel $E[I]$. Define 
    \[I\ast E:= E/E[I].\]
\end{definition}

Each ideal class contains an integral ideal representative which is of norm coprime to $p$. The principal ideals are generated by a single endomorphism, and so act trivially. In \cite[\S3]{Wat69}, Waterhouse establishes that, in the case where $E$ is ordinary, this is a free and transitive group action on the set of elliptic curves with endomorphism ring isomorphic to $\End(E)$. The proof goes through showing that ideals of $\End(E)$ satisfy certain properties qualifying them as \emph{kernel ideals}.

In the supersingular case, the situation is more complicated. If $E$ is supersingular, then $\End(E)$ is a noncommutative ring in a quaternion algebra. In particular, $\End(E)$ does not have a class \emph{group} of (left or right) ideals. There is a partial remedy and a full remedy. The partial remedy: if $k = \mathbb{F}_p$, then $\End_k(E)$ is isomorphic to an order in the imaginary quadratic field $\Q(\sqrt{-p})$. 
In this case, one can work with the class group of $\End_k(E)$ and use Definition~\ref{def:ideal_action}, just as in the ordinary case. The group action will be free and transitive (modulo a minor subtlety highlighted in the proof of~\cite[Thm.\,4.5]{Schoof_NonsingPlaneCubics}). However, if $k = \F_{p^n}$ for $n>1$ or $k = \overline{\F_p}$, we find ourselves in need of additional framework: orientations on supersingular elliptic curves are the full remedy. The remainder of this section deals with the supersingular case. 

Let $k \subset \overline{\F_p}$ and consider a supersingular elliptic curve $E/k$. 
By \cite{Silverman}, the endomorphism ring $\End(E)$ is isomorphic to a maximal order in the quaternion algebra $B_{p,\infty}:=\End(E)\otimes_\mathbb{Z}\mathbb{Q}$ ramified precisely at $p$ and $\infty$. 
Any non-scalar element $\alpha\in\End(E)\setminus\mathbb{Z}$ generates an imaginary quadratic order. Let $K$ be an imaginary quadratic field in which $p$ does not split. 
This condition gives the existence of an embedding of $K$ into $B_{p,\infty}$ \cite[Prop. 14.6.7]{Voight}.

\begin{definition}\label{def:Korientation}
    A \textit{$K$-orientation} on an elliptic curve $E$ is an embedding 
    \[\iota: K\hookrightarrow\End(E)\otimes_\mathbb{Z}\mathbb{Q}.\]
    For an order $O$ of $K$ such an embedding is called a primitive $O$-orientation if $\iota(O) = \iota(K)\cap \End(E)$.
    The pair $(E,\iota)$ is called a primitively $O$-oriented supersingular elliptic curve. 
    We denote by $\Ell_k(O)$ the set of primitively $O$-oriented elliptic curves over $k$. 
\end{definition}

\begin{remark}
When $E$ is supersingular as above, a $K$-orientation $\iota$ maps into a four-dimensional $\Q$-algebra.
In the case where the endomorphism ring of $E$ is commutative, Definition~\ref{def:Korientation} can still apply: the map $\iota$ defines an isomorphism $O\cong \End(E)$. Note that in both cases we call such a map a primitive $O$-orientation on $E$, to unify notation. 
\end{remark}

\begin{definition}[$K$-oriented isogeny]
A $K$-oriented isogeny is an isogeny $\varphi:(E_0,\iota_0)\to (E_1,\iota_{\varphi})$ between $K$-oriented elliptic curves such that $\varphi:E_0\to E_1$ as an isogeny of elliptic curves and $\iota_{\varphi}(-) = \frac{1}{\deg\varphi}\varphi\circ\iota_0(-)\circ \widehat{\varphi}$. 
\end{definition}
An isomorphism of elliptic curves is likewise a $K$-oriented isomorphism if it is of degree-1 and satisfies the above properties.

Via the Deuring lifting theorem, the theory of oriented supersingular elliptic curves is closely related to the theory of CM elliptic curves.

\begin{definition}
    There is an extension $L'$ of the ring class field $L$ of $O$ and a prime $\mathfrak{p}$ above $p$ in $O_{L'}$ such that every elliptic curve with CM by $O$ has a representative defined over $L'$ with good reduction at $\mathfrak{p}$. 
    Let $\Ell_{L'}(O)$ denote the set of isomorphism classes of elliptic curves with endomorphism ring isomorphic to $O$ and having good reduction over $\mathfrak{p}$. 
    Let $\rho:\Ell_{L'}(O)\to \Ell_k(O)$ denote the reduction map modulo $\mathfrak{p}$.
\end{definition}

\begin{definition}\label{def:group_action_orientedcurves}
    Let $(E,\iota)\in\Ell_k(O)$ and take an ideal $\afrak$ of $O$. Define the (group-theoretic) intersection:
    \[E[\iota(\afrak)] := \bigcap_{\alpha\in\afrak}\ker(\iota(\alpha)).\]
    Let $\varphi_{\afrak}$ denote a $K$-oriented isogeny of $(E,\iota)$ with kernel $E[\iota(\afrak)]$. 
    Such an isogeny $\varphi_{\afrak}:(E,\iota)\to (E_{\afrak},\iota_{\afrak})$ is unique up to $K$-oriented automorphism on the codomain.
    \\When $\#E[\iota(\afrak)] = N(\afrak)$, we define the action of $\afrak$ on $(E,\iota)$ to be:
    $\afrak *(E,\iota) = (E_{\afrak},\iota_{\afrak})$.
\end{definition}

Definition~\ref{def:group_action_orientedcurves} is precisely the supersingular analogue of Definition~\ref{def:ideal_action} that we need. The principal ideals $(\beta)$ of $O$ are generated by endomorphisms of $E$, and thus act trivially on $(E,\iota)$ since $\beta\in O$ commutes with the image of $\iota$ in $\End(E)$. The following theorem of Onuki completes the picture by providing the supersingular analogue of Theorem~\ref{thm:waterhouse}.

\begin{theorem}[{\cite[Theorem 3.4]{onuki}}] \label{onuki}
    When $p$ is not split in the imaginary quadratic field $K$ and $p$ is coprime to the conductor of the order $O$ of $K$, then Definition~\ref{def:group_action_orientedcurves} gives a free and transitive action of $\cl_O$ on $\rho(\Ell_{L'}(O))$.
\end{theorem}

To understand when this gives a free and transitive action on the set $\Ell_k(O)$ of primitively $O$-oriented elliptic curves, we need to understand the relationship between the sets $\rho(\Ell_{L'}(O))$ and its superset $\Ell_k(O)$:

\begin{corollary}[{\cite[Theorem 4.5]{Schoof_NonsingPlaneCubics},\cite[Theorem 4.4]{orientationsandcycles}}]
    If $p$ is ramified in the imaginary quadratic field $O\otimes_\mathbb{Z}\mathbb{Q}$, then $\cl_O$ acts freely and transitively on the set of primitively $O$-oriented supersingular elliptic curves over $\overline{\mathbb{F}_p}$.
    If $p$ is inert in the imaginary quadratic field $O\otimes_\mathbb{Z}\mathbb{Q}$, then $\cl_O$ has two orbits in the set of primitively $O$-oriented supersingular elliptic curves over $\overline{\mathbb{F}_p}$. 
\end{corollary}
Concretely, if $p$ is not inert in the imaginary quadratic field $K$ containing $O$, then $\cl_O$ acts freely and transitively on the set of isomorphism classes of primitively $O$-oriented elliptic curves $\Ell_k(O)$. If $p$ is split in $K$, $\Ell_k(O)$ is a set of ordinary elliptic curves. If $p$ is ramified, $\rho(\Ell_{L'}(O)) = \Ell_k(O)$ is a set of primitively $O$-oriented supersingular elliptic curves. If $p$ is inert, $\cl_O$ acts on $\rho(\Ell_{L'}(O))$, which is again a set of primitively $O$-oriented supersingular elliptic curves.

\section{Generalized class group actions} \label{sec:gen_CM_action}


Let $K$ be an imaginary quadratic number field and let $O$ be an order in $K$. Let $k \subset \overline{\F_p}$. In this section, if $p$ is inert in $K$, we fix the orbit $\rho(\Ell_{L'}(O))$ of $\cl_O$ and call it $\Ell_k(O)$ by an abuse of notation. 

Let $\mfrak$ be a modulus in $O$. Let $H$ be a congruence subgroup for $\mfrak$ that is contained in $P_O(\mfrak)$. Let
\[ \cl_H = I_O(\mfrak)/H \] 
be the corresponding generalized class group. Because $H \subseteq P_O(\mfrak)$ the map
\[ \cl_H \times \,  \Ell_k(O) \to \Ell_k(O) : ([\afrak], E) \mapsto \varphi_{\afrak}(E) = E / E[\afrak] \]
remains a well-defined group action. However, if $H \subsetneq P_O(\mfrak)$ then this no longer yields a free group action: the class of any ideal $\afrak \in P_O(\mfrak) \setminus H$ is a non-trivial element acting trivially. This creates room for an action of $\cl_H$ on elements of $\Ell_k(O)$ equipped with extra data, i.e., with \emph{$\mfrak$-level structure}, which we now define. 

\subsection{$\mfrak$-level structures}

Our starting observation is:

\begin{lemma}\label{lem:structure-m-torsion}
Let $O$ be an imaginary quadratic order and let $E \in \Ell_k(O)$. Suppose $\mfrak \subseteq O$ is a proper ideal of norm coprime with $p$.
Then
\[E[\mfrak]\cong O/\mfrak\]
as $O$-modules; in particular, they are are also isomorphic as groups.
\end{lemma}
\begin{proof}
  If $O$ is a maximal order, then for $k \subseteq \C$ this is precisely~\cite[Proposition II.1.4(b)]{SilvermanII}, while for $k$ finite one can use the Deuring lifting theorem to reduce to the case $k \subseteq \C$. 
  To obtain the statement for arbitrary imaginary quadratic orders and over arbitrary base fields, 
  we start by writing $\mfrak = N O + \alpha O$ with
  $N$ the positive generator of $\mfrak \cap \Z$. We can assume that the norm of $\alpha$ is coprime with $p$. Indeed, this follows from the fact that $N(\mfrak)$
  equals the greatest common divisor of the norms of the elements of $\mfrak$ (because it is a proper ideal).

  
  
  We rely on a result due to Lenstra~\cite[Lemma~2.1]{lenstra_CM}, stating that the lemma is true for principal ideals. Applying this to $\alpha O$, we obtain an isomorphism
  \[ \phi : \frac{O}{\alpha O} \longrightarrow E[\alpha] \]
  of $O$-modules. From this it readily follows that $O/\mfrak \cong E[\mfrak]$ as groups, because for any finite abelian group $A$ we have $A/NA \cong A[N]$. However, our goal is to establish this isomorphism at the level of $O$-modules. 

Proving this amounts to showing that for any $\delta$ such that $O = \Z[\delta]$ we can find a point $P \in E[\mfrak]$ such that
\[
E[\mfrak] = \langle P, \delta(P) \rangle.
\]
Indeed, this gives a well-defined map of $O$-modules given by
\[ O/\mfrak \to E[\mfrak] : 1 \mapsto P, \]
which is clearly a surjective group homomorphism, hence bijective because we already know an isomorphism as groups.

We again use the fact that $\mfrak$ is proper, which implies that $\bar{\mfrak} \mfrak = N(\mfrak)O$. 
We start with the curve $E_{\mfrak} := E/E[\mfrak]$. Note that $E_{\mfrak} \in \Ell_k(O)$ since the corresponding isogeny $\varphi_{\mfrak}$ is horizontal, so by abuse of notation we can view $\delta$ as an endomorphism of $E_{\mfrak}$. 
Again by Lenstra's theorem, there exists an isomorphism of $O$-modules
\[
O/N(\mfrak)O \rightarrow E_{\mfrak}[N(\mfrak)].
\]
Set $P \in E_{\mfrak}[N(\mfrak)]$ to be the image of $1$ under this isomorphism. This implies that $E_{\mfrak}[N(\mfrak)] = \langle P, \delta(P) \rangle.$
Then
\[
E[\mfrak] = \langle \varphi_{\bar{\mfrak}}(P), \varphi_{\bar{\mfrak}}(\delta(P)) \rangle = \langle \varphi_{\bar{\mfrak}}(P), \delta(\varphi_{\bar{\mfrak}}(P)) \rangle,
\]
where the first equality follows from the fact that $\varphi_{\bar{\mfrak}} = \widehat{\varphi_{\mfrak}}$, and the second follows from the fact that $\delta$ ``commutes" with horizontal isogenies. Thus we have constructed our basis in the correct form. \qed
\end{proof}

\begin{remark} 
Lemma~\ref{lem:structure-m-torsion} covers all cases of interest to this paper, but it holds more generally: e.g., from the proof, it is immediate that the statement remains true if $\mfrak$ is principal and generated by a separable endomorphism, or whenever $\mfrak$ contains a separable element and $E[\mfrak]$ is cyclic. However, not all conditions can be dropped: one clearly runs into issues as soon as $\varphi_{\mfrak}$ is inseparable, while a more subtle counterexample is $O = \Z[\ell \sqrt{-1}]$ and $\mfrak$ the $O$-ideal generated by $\ell^2$ and $\ell \cdot \ell \sqrt{-1}$ (where $\ell$ denotes any prime number different from $p$).
\end{remark}


Applying Lemma~\ref{lem:structure-m-torsion} to $\mfrak = NO$ for some integer $N$ coprime to $\charac k$, we recover the well-known fact that
\[ E[N] \cong O / NO \cong \Z / N\Z \times \Z / N\Z \]
as groups, where upon writing $O = \Z[\sigma]$ for some generator $\sigma$, an instance of the last isomorphism is given by $1 \mapsto (1, 0)$, $\sigma \mapsto (0, 1)$.
This
motivates the following generalization of Definition \ref{def:level-structure-classic}.

\begin{definition}\label{def:level-structure-general}
   Let $\mfrak \subseteq O$ be a proper ideal coprime to $\charac k$. Let $\Gamma \subseteq \GL(O/\mfrak)$ be a subgroup and let $E$ be an elliptic curve primitively oriented by $O$. A $\Gamma$-level structure on $E$ is then a group isomorphism
   \[
   \Phi : O/\mfrak \rightarrow E[\mfrak]
   \]
   defined up to pre-composition with an element $\gamma \in \Gamma$ and post-composition with a $K$-oriented automorphism. We denote by $Y_\Gamma$ the set of primitively $O$-oriented elliptic curves equipped with a $\Gamma$-level structure, up to $K$-oriented isomorphisms. If $\Gamma$ consists of $O$-module automorphisms, then we denote by $Z_\Gamma \subseteq Y_\Gamma$ the subset for which the level structure is  an isomorphism of $O$-modules. 
\end{definition}

\noindent   The reason for highlighting the subset $Z_\Gamma$ will become apparent in Section~\ref{ssec:family}.


\begin{remark} \label{rmk:ignoreunits}
  Considering $\Gamma$-level structures up to $K$-oriented automorphisms amounts to identifying $(E, \Phi)$ and $(E, \iota(u) \circ \Phi)$ for every $u \in O^\times$; here $\iota$ denotes the implicit embedding of $O$ in $\End(E)$. However, in most cases this can be ignored because it is already taken care of by the $\Gamma$-level structure.  E.g., this is true if $O^\times = \{ \pm 1\}$ and
  $\Gamma$ is closed under negation.
\end{remark}

More concretely, in view of the lemma below, defining a $\Gamma$-level structure amounts to specifying a point $P$ of order $a_{\mfrak}$ and a point $Q$ of order $b_{\mfrak}$ such that $\frac{a_{\mfrak}}{b_{\mfrak}}P, Q$ is a basis of $E[b_{\mfrak}]$, considered up to ``base changes" as specified by the subgroup $\Gamma$.

\begin{lemma} \label{lem:m-torsion}Let $\mfrak \subseteq O$ be a proper ideal. Then there exist unique $a_{\mfrak}, b_{\mfrak} \in \Z$ such that $\charac k \nmid b_{\mfrak} \mid a_{\mfrak}$ and
\[E[\mfrak] \cong \frac{\Z}{(a_{\mfrak})} \times \frac{\Z}{(b_{\mfrak})} \]
for all $E \in \Ell_k(O)$.
\end{lemma}

\begin{proof}
    This is standard: every finite subgroup of $E$ admits such a decomposition, and the independence of $E$ follows because any two such curves are connected by a horizontal isogeny of norm coprime with $\mfrak$. \qed
\end{proof}

In order to motivate the next sections, let us conclude by restating (a slightly extended version of) the observation made by Galbraith, Perrin, and Voloch \cite{GPV} in the context of CSIDH.
Here one considers $\mfrak = NO$ and $\Gamma = \{ \id \}$; for simplicity we will just write $Y_N, Z_N$ instead of $Y_{\{\id\}}, Z_{\{\id\}}$. Putting an $\{\id\}$-level structure on a curve $E \in \Ell_k(O)$ just amounts to choosing a basis $P, Q \in E[N]$, i.e., a full level-$N$ structure. Writing $O = \Z[\sigma]$, elements of $Z_N \subseteq Y_N$ correspond to bases of the form $P, \sigma(P)$; it is a consequence of Lemma~\ref{lem:structure-m-torsion} that such bases indeed exist. 

\begin{theorem} \label{thm:GPV}
Let $N$ be a positive integer coprime to $\charac k$. Then the ray class group
  \[ \cl_{O,1}(NO) = I_O(NO) / P_{O, 1}(NO) \]
  acts freely on both $Y_N$ and $Z_N$; in the latter case, the action is also transitive.
\end{theorem}

\begin{proof} This is a special case of Theorem~\ref{thm:maxlevel} below. \qed
\end{proof}

\subsection{A family of congruence subgroups} \label{ssec:family}

The goal of this section (and of this paper) is to embed Theorem~\ref{thm:GPV} in a more general story. 
We concentrate on congruence subgroups of the form 
\[ P_{O, \Lambda}(\mfrak) = \{ \, \alpha O \, | \, \alpha \in K^\times \text{ and } \alpha \equiv \lambda \bmod \mfrak \text{ for some $\lambda \in \Lambda$ coprime to $N(\mfrak)$} \, \},\]
where $\Lambda$ is a multiplicatively closed subset of $O$. This covers the aforementioned congruence subgroups $P_{O, \Z}(fO)$ and $P_{O, 1}(\mfrak) = P_{O, \{ 1\}}(\mfrak)$ as special cases, yet it also introduces several interesting new examples. 
\begin{remark} \label{rmk:includeunitsornot}
Notice that $\Lambda$ and $\pm \Lambda$ or more generally $O^\times \Lambda$ 
define the same congruence subgroup, as one can always change the generator $\alpha$ of a principal ideal accordingly.
Thus it would make sense to impose $O^\times \subseteq \Lambda$. However, we refrain from doing this, in order to keep covering standard notation such as $P_{O, \Z}(fO)$; also the exact sequence from Proposition~\ref{prop:exact-sequence-general} is affected by this, see Remark~\ref{rmk:sequence_changes}.
\end{remark}

Now, to such a congruence subgroup $P_{O, \Lambda}(\mfrak)$ we can naturally associate the subgroup 
\[
\Gamma_{O, \Lambda}(\mfrak) = \{
\, \mu_\alpha \,
\mid \alpha O \in P_{O, \Lambda}(\mfrak) \,
\} = \{ \, \mu_\lambda \, | \, \lambda \in O^\times \Lambda \, \} \subseteq \GL(O/\mfrak)
\]
where $\mu_\alpha$
refers to the action of multiplication by $\alpha$ on ${O/\mfrak}$. By definition of $P_{O, \Lambda}(\mfrak)$, this is a multiplicative subset of the finite group $\GL(O/\mfrak)$, hence indeed a subgroup.
Note that the $\mu_\alpha$'s are $O$-module automorphisms, so both  $Y_{\Gamma_{O, \Lambda}(\mfrak)}$ and
$Z_{\Gamma_{O, \Lambda}(\mfrak)}$
are well-defined.

\begin{theorem} \label{thm:maxlevel}
  Let $\mfrak \subseteq O$ be a proper ideal, and let $H = P_{O, \Lambda}(\mfrak)$ be as above. Then 
 \begin{equation} \label{eq:our_action} 
    [\afrak] \star (E, \Phi) = (\varphi_{\mathfrak{a}}(E), \varphi_{\mathfrak{a}} \circ \Phi)
\end{equation}  
  is a well-defined free action
 of $\cl_H$ on $Z_{\Gamma_{O, \Lambda}(\mfrak)}$. Moreover, this action is transitive. If $\Lambda \subseteq O^\times \Z$ then this extends to a free action of $\cl_H$ on $Y_{\Gamma_{O, \Lambda}(\mfrak)}$.
\end{theorem}

\begin{proof}
Since $\deg \varphi_{\afrak} = N(\afrak)$ is assumed coprime with $\mfrak$, it follows readily that the right-hand side of~\eqref{eq:our_action} is an element of $Y_{\Gamma_{O, \Lambda}(\mfrak)}$. Using that $\varphi_{\afrak}$ is $K$-oriented, we also see that it concerns an element of 
$Z_{\Gamma_{O, \Lambda}(\mfrak)}$ as soon as $(E, \Phi)$ is.

Now assume $(E, \Phi) \in Z_{\Gamma_{K, \Lambda}(\mfrak)}$ and
let $\afrak = \alpha O$ be the principal ideal generated by some $\alpha \in O$. Then 
\begin{equation} \label{eq:compatibility} \varphi_{\afrak} \circ \Phi = \Phi \circ \mu_\alpha 
\end{equation}
because $\Phi$ is an isomorphism of $O$-modules. It follows that $\Phi$ and $ \varphi_{\afrak} \circ \Phi$ define the same $\Gamma_{O, \Lambda}(\mfrak)$-level structure on $E$ if and only if $\alpha O \in P_{O, \Lambda}(\mfrak)$. But this implies that the action is well-defined and free. As for the transitivity, it suffices to argue that if
\[ \Phi_1, \Phi_2 : O/\mfrak \to E[\mfrak] \]
are two isomorphisms as $O$-modules, then there exists $\alpha \in O$ such that $\Phi_2 = \varphi_{\alpha O} \circ \Phi_1$. This is evident from the fact that we are dealing with free rank-$1$ modules over $O/\mfrak$.

Finally, we need to show that if $\Lambda \subseteq O^\times \Z$, then we still have a well-defined and free action on all of $Y_{\Gamma_{O, \Lambda}(\mfrak)}$. By ignoring post-compositions with $K$-oriented automorphisms, we can in fact assume $\Lambda \subseteq \Z$. For this we need to show that
\[ \varphi_{\alpha O} \circ \Phi = \Phi \circ \mu_{\alpha'} \]
for some $\alpha'O \in P_{O, \Lambda}(\mfrak)$ if and only if $\alpha O \in P_{O, \Lambda}(\mfrak)$. Since we are working modulo $\mfrak$, this amounts to saying that
\[ \varphi_{\alpha O} \circ \Phi = \Phi \circ \mu_{\lambda} = [\lambda] \circ \Phi \]
for some $\lambda \in \Lambda$ if and only if $\alpha O \in P_{O, \Lambda}(\mfrak)$; the last equality follows because $\Phi$ is a group homomorphism. If $\alpha O \in P_{O, \Lambda}(\mfrak)$ then the existence of such a $\lambda$ is clear. On the other hand, if such a $\lambda$ exists then from~\cite[Cor.\,III.4.11]{Silverman} it follows that $\alpha \equiv \lambda \bmod \mfrak$, as wanted. (Note that, in the above reasoning, we have used that we can ignore units, in view of Remark~\ref{rmk:ignoreunits}.) \qed
\end{proof}

In general, the action of the generalized class group $\cl_H$ on $Y_{\Gamma_{O, \Lambda}(\mfrak)}$ is far from transitive. E.g., recall from Theorem~\ref{thm:GPV} that 
the ray class group acts freely on
\[ Y_{\Gamma_N} = \{ \, (E, P, Q) \, | \, E \in \Ell_k(O), \ \text{$P, Q$ basis of $E[N]$} \, \}, \]
but when writing $O = \Z[\sigma]$, it is easy to see that if $P$ happens to be an eigenvector of $\sigma$, this can never be ``undone" by acting with a ray class. There are two natural ways to make the action more transitive:
\begin{itemize}
    \item Restricting to a subset of $Y_{\Gamma_{O, \Lambda}(\mfrak)}$; this is exactly what we did above when studying $Z_{\Gamma_{O, \Lambda}}$, which seems to be the most natural option.
    \item Further identifying elements of $Y_{\Gamma_{O, \Lambda}}$ by working with a bigger group $\Gamma \supseteq \Gamma_{O, \Lambda}$.
\end{itemize}

We now  analyse the action of $\cl_H$ on a set defined by $\Gamma \supseteq \Gamma_{O, \Lambda}(\mfrak)$. First, note that we are free to chose any such set, as the following lemma shows.

\begin{lemma}\label{lem:supergroupaction}
     Assume $\Lambda \subseteq \Z$, let $H = P_{O, \Lambda}(\mfrak)$ and consider the free action of $\cl_H$ on $Y_{\Gamma_{O, \Lambda}}$ from above. Then this descends to a well-defined action
 of $\cl_H$ on $Y_{\Gamma}$ for any $\Gamma \supseteq \Gamma_{O, \Lambda}(\mfrak)$.
\end{lemma}
\begin{proof}
    The set $Y_{\Gamma}$ consists of equivalence classes of elements of $ Y_{\Gamma_{O, \Lambda}}$. 
    Thus, we only need to show that if $(E, \Phi) \sim (E, \Phi')$, i.e. $\Phi' = \Phi \circ T$ for some $T \in \Gamma$, then $(\varphi_{\mathfrak{a}}(E), \varphi_{\mathfrak{a}} \circ \Phi) \sim (\varphi_{\mathfrak{a}}(E), \varphi_{\mathfrak{a}} \circ \Phi')$ for all $[\mathfrak{a}]\in \cl_H$. 
    But this is clearly true, since we still have $\varphi_{\mathfrak{a}} \circ \Phi' = \varphi_{\mathfrak{a}} \circ \Phi \circ T$. \qed
\end{proof}

\subsection{A generalised exact sequence}
Recall that $H \subseteq P_O(\mfrak)$. Thus, the class group $\cl_H$ surjects onto $\cl_O$, and as the action of $\cl_O$ is well understood, we aim to study the action of the kernel of this surjection. To do this, we start by slightly generalising the exact sequence from Theorem~\ref{thm:kopplagarius}:

\begin{proposition}\label{prop:exact-sequence-general} Let $\mfrak, \Lambda$ and $O$ be as above. Let $H = P_{O, \Lambda}(\mfrak)$. Then $\Lambda$ defines a subgroup of $(O/\mfrak)^\times$, defined as $\Delta := \phi(\Lambda) \cap (O/\mfrak)^\times$, where $\phi$ denotes the natural surjection from $O$ to $O/\mfrak$. Then, there is an exact sequence
    \[
    1 \rightarrow O^\times/(O^\times \cap (\Lambda + \mfrak)) \rightarrow (O/\mfrak)^\times/\Delta \rightarrow \cl_H \rightarrow \cl_O \rightarrow 1    
    \]
\end{proposition}
\begin{proof}
    The proof closely follows the proof from Cox \cite[Theorem 7.24]{cox}, which proves the special case where $\Lambda = \Z$, $\mfrak = (f)$ and $O = O_K$.
    
    We prove this from the right to left. The surjection $\pi: \cl_H \rightarrow \cl_O$ is obtained from the natural map sending $[\afrak] \in I_O(\mfrak)/P_{O, \Lambda}(\mfrak)$ to the class of $\afrak$ in $I_O(\mfrak)/P_{O}(\mfrak) \cong \cl_O$. The kernel is therefore exactly
    $P_O(\mfrak)/P_{O, \Lambda}(\mfrak)$. Next, we show that there is a surjection
    \[
    (O/\mfrak)^\times/\Delta \rightarrow P_O(\mfrak)/P_{O, \Lambda}(\mfrak)
    \]
    obtained by sending $[[\alpha]] \in (O/\mfrak)^\times/\Delta$ to $\alpha O$ (we will use the notation $[\gamma]$ for elements of $(O/\mfrak)^\times$, and $[[\gamma]]$ for elements of $(O/\mfrak)^\times/\Delta$). The ideal $\alpha O$ is clearly in $P_O(\mfrak)$. Further, let $[\alpha] = [\beta][\delta]$, for some $\delta \in \Lambda$, i.e. $\alpha$ and $\beta$ are in the same class of $(O/\mfrak)^\times/\Delta$. Then, unraveling the definitions, there exists some $u \in O$ such that $u\alpha \equiv u\beta \delta \equiv 1 \pmod{\mfrak}$. Further, we can choose some $\delta' \in \Lambda$ such that $[\delta'] \equiv [\delta]^{-1}$. Thus, we have that
    \[
    \alpha O \cdot u\beta\delta O = \beta O \cdot u\alpha \delta^{-1} O,
    \]
    which shows that the map is a well defined group homomorphism, since $u\beta\delta O \in P_{O, 1}(\mfrak) \subseteq P_{O, \Lambda}(\mfrak)$, and $u\alpha \delta^{-1} O \in P_{O, \Lambda}(\mfrak)$.

    Next, we show that the map is surjective. Let $\gamma O \in P_O(\mfrak)$. Obviously, if $\gamma \in O$, then $[[\gamma]]$ maps to $\gamma O$. In general, $\gamma$ can be written as $\gamma_1 \gamma_2^{-1}$ for $\gamma_1, \gamma_2 \in O$, which both are coprime to $\mfrak$. Let $N$ be the norm of $\gamma_2$. Since $N$ is coprime to $\mfrak$, there exists a $k \in \Z$ such that $kN \equiv 1 \mod \mfrak$. Since $N\gamma_2^{-1} = \overline{\gamma_2}$, we have that the class $[[\gamma_1][k\overline{\gamma_2}]]$ maps to $k\gamma_1\overline{\gamma_2}O = \gamma O \cdot kN O$, and since $kN \equiv 1 \mod \mfrak$, we have $kN O \in P_{O, 1}(\mfrak) \subseteq P_{O, \Lambda}(\mfrak)$, proving that the map is surjective.

    Finally, assume $[[\alpha]] \in (O/\mfrak)^\times/\Delta$ satisfies $\alpha O \in P_{O, \Lambda}(\mfrak)$. Thus, we have that $\alpha O = \beta\gamma^{-1} O$, for some $\beta, \gamma$ satisfying $[\beta'][\gamma]^{-1} \in \Delta$, and that $\alpha = \mu\beta\gamma^{-1}$ for some $\mu \in O^\times$. This in turn means that $[\alpha] = [\mu][\beta][\gamma]^{-1}$, and since $[\beta][\gamma]^{-1} \in \Delta$, we see that $[[\alpha]] = [[\mu]]$, i.e. $[[\alpha]]$ is in the image of the natural homomorphism $O^\times \rightarrow (O/\mfrak)^\times/\Delta$, whose kernel is in turn exactly $O^\times \cap (\Lambda + \mfrak)$. \qed 
\end{proof}

\begin{remark}
    We can compare Proposition \ref{prop:exact-sequence-general} with both the classical $O=O_K$ version of Theorem \ref{thm:kopplagarius} and the formula for computing the size of the ray class group from \cite[Theorem 3.2.4]{cohen2012advanced}: Specialising to $\Lambda = \Z$ and $\mfrak = (f)$, and writing $O := \Z + fO_K$, we immidiately see that $$O_K^\times \cap (\Lambda + \mfrak) = O_K^\times \cap (\Z + fO_K) = O^\times,$$ and similarily, $\Delta = \phi(\Z) \cap (O_K/fO_K)^\times = (\Z + fO_K/fO_K)^\times = (O/fO_K)^\times$, recovering the exact sequence from Theorem~\ref{thm:kopplagarius}. 
    
    The size of the ray class group can be computed from the exact sequence when $\Lambda = \{1\}$, in which case one finds that $\Delta = \{[1]\}$, and thus
    \[
    \#\cl_H = h(O_K)\frac{\#(O_K/\mfrak)^\times}{\frac{O_K^\times}{O_K^\times \cap (1 + \mfrak)}} = h(O_K)\frac{\#(O_K/\mfrak)^\times}{[O_K^\times : O_{K, \mfrak}^\times]}
    \]
    where $O_{K, \mfrak}^\times$ is the group of units congruent to $1 \bmod{\mfrak}$. This same remark applies in the relative case, $O'\subseteq O\subseteq O_K$.
\end{remark}

\begin{remark} \label{rmk:sequence_changes}
  Recall from Remark~\ref{rmk:includeunitsornot} that switching from $\Lambda$ to $O^\times \Lambda$ does not affect the congruence subgroup $P_{O, \Lambda}(\mfrak)$, and therefore it does not change the generalized class group either. Also, it does not affect the subgroup $\Gamma_{O, \Lambda}(\mfrak)$. However, it is interesting to observe that it can slightly reorganize the terms in the exact sequence from Proposition~\ref{prop:exact-sequence-general}. Indeed, switching from $\Lambda$ to $O^\times \Lambda$ has the effect of 
  folding the exact sequence, which is of the form 
  \[ 1 \to G_1 \to G_2 \stackrel{f}{\to} G_3 \to G_4 \to 1, \qquad \text{into} \qquad 1 \to \frac{G_2}{\ker f} \to G_3 \to G_4 \to 1. \]
\end{remark}

Since we know that $\cl_O$ acts freely on the set of primitively $O$-oriented curves, we use the surjection $\pi : \cl_H \rightarrow \cl_O$ from the exact sequence above, and study the action of $\ker \pi$. By Proposition \ref{prop:exact-sequence-general}, The elements of $\ker \pi$ are principal ideals, which can be identified by elements of $O/\mfrak$ up to multiplication by $\Delta$ and $O^\times$. In particular, they are endomorphisms, leaving the curve fixed, and acting on the different $\Gamma$-level structures, which in turn can be identified (by definition) with left cosets of $\Gamma \subset \GL(O/\mfrak)$, up to $K$-oriented isomorphisms. This action is fairly easy to describe explicitly, as the following lemma shows.

\begin{corollary} \label{cor:kernel-action} 
    Let $O, \mfrak, \Lambda \subseteq \Z$ and the map $\pi$ be as before. Let $\Gamma \supseteq \Gamma_{O, \Lambda}(\mfrak)$. Then, $\ker \pi$ acts on the set
    \[
    X_\Gamma := \{M\Gamma \mid M \in \GL(O/\mfrak)\}/\sim
    \]
    where $\sim$ is the equivalence relation obtained by identifying cosets up to left multiplication by $\mu_u$ for $u \in O^\times$.
\end{corollary}
\begin{proof}
    As we have seen, $\ker \pi$ can be identified with elements $[\alpha] \in (O/\mfrak)^\times$, up to multiplication by $\Lambda$ and $O^\times$. We show that the natural action of sending the left coset $M\Gamma $ to $\mu_\alpha M\Gamma$ is well defined. This is clearly a well defined action by $(O/\mfrak)^\times$, so it suffices to show that multiplication by elements of $\Lambda$ and $O^\times$ act trivially. Since $\lambda \in \Lambda \subseteq \Z$, it is clear that
    \[
    \mu_\lambda M \Gamma = M \mu_\lambda \Gamma = M \Gamma,
    \]
    and further, for $u \in O^\times$, $\mu_u$ acts trivially by definition of $X_\Gamma$. \qed
\end{proof}

\subsection{Suborder class group actions} \label{sec:orders_acting}

As one of our main examples, let us concentrate on the case
where $\Lambda = \Z$ and $\mfrak = fO$ for some prime number $f$ different from $\charac k$. Pick $\sigma \in O$ such that $O = \Z[\sigma]$. Write $H = P_{O, \Z}(fO)$ and observe that 
$\Gamma := \Gamma_{O, \Z}(fO) \subseteq \GL(O / fO)$ 
is just the group of multiplications $\mu_\lambda$ by an integer $\lambda$ that is not divisible by $f$.
Thus we have
\[ Y_\Gamma = \{ \, (E, P, Q) \, | \, E \in \Ell_k(O), \, P, Q \text{ basis of } E[f] \} / \sim  \]
where $(E, P, Q) \sim (E, \lambda P, \lambda Q)$ for any scalar $\lambda \in (\Z/f\Z)^\times$. The action of $\cl_H$ on $Y_\Gamma$ is not transitive. Recall that there are two approaches towards turning this into a ``more transitive" action:
\begin{itemize}
    \item Instead of $Y_\Gamma$, we can act on $Z_\Gamma$, i.e., we can require that the isomorphism
    \[ \Phi : O/fO \to E[f] \]
    is an isomorphism of $O$-modules. This amounts to picking a basis of $E[f]$ of the form $P, \sigma(P)$. Note that it suffices to specify $P$ in this case, and because of the scaling we are in fact specifying a cyclic subgroup $C \subseteq E$ of order $f$. However, since $P, \sigma(P)$ must be a basis, the subgroups we thus obtain are those that are \emph{not} eigenspaces of $\sigma$ acting on $E[f]$. In other words, we can identify
    \begin{equation} \label{eq:ZGamma_desc} Z_\Gamma = \{ \, (E, C) \, | \, E \in \Ell_k(O), \, C \subseteq E \text{ kernel of descending $f$-isogeny} \}.
    \end{equation}
    Thanks to Theorem~\ref{thm:maxlevel} we know that $\cl_H$ acts freely and transitively on this set.
    \item Alternatively, we can apply the idea of making the action of $\cl_H$ on $Y_\Gamma$ more transitive by enlarging $\Gamma$. In this case it is natural to consider 
    \[ \Gamma_N^0 = \left\{ \, \begin{pmatrix} \ast & \ast \\ 0 & \ast \end{pmatrix} \, \right\} \supseteq \Gamma, \]
where the last inclusion makes sense upon identification of $\GL(O/fO)$ with $(\Z/f\Z)^2$. Thus by Lemma~\ref{lem:supergroupaction} we also have a natural action of $\cl_H$ on
\[ Y_N^0 = Y_{\Gamma^0_N} = \{ \, (E, C) \, | \, E \in \Ell_k(O), \, C \subseteq E \text{ cyclic subgroup of order $f$} \, \}. \]
However, unless $f$ is inert, this action is neither free nor transitive: any eigenspace of $\sigma$ acting on $E[f]$ is fixed by every element of $\cl_H$. To turn this into a free and transitive action one has to discard the eigenspaces; as such one again arrives at $Z_\Gamma$.
\end{itemize}

Now let $O' = \Z + fO$ be a suborder of relative conductor $f$ and recall from Theorem~\ref{thm:clOisgeneralized} that the natural map
\begin{equation} \label{eq:classgroupiso} \cl_{O'} \to \cl_H : [\afrak] \mapsto [\afrak O] \end{equation}
is an isomorphism. In fact, more generally, it is easy to check that the exact sequence from Proposition~\ref{prop:exact-sequence-general} fits in an isomorphism of exact sequences
\[ \begin{array}{ccccccccccc}
1 & \to & \frac{O^\times}{O'^\times} & \to & \frac{(O/fO)^\times}{(O'/fO)^\times} & \to & \cl_{O'} & \to & \cl_O & \to & 1 \\
& & \downarrow & & \downarrow & & \downarrow & & \downarrow &  \\ 
1 & \to & \frac{O^\times}{O^\times \cap (\Lambda + fO)} & \to & \frac{(O/fO)^\times}{\Delta} & \to & \cl_H & \to & \cl_O & \to & 1 
\end{array} \]
where the vertical maps are the natural maps, and where the sequence on top is the exact sequence from Theorem~\ref{thm:kopplagarius}.

Now recall from Section~\ref{ssec:CMaction} that we have a free and transitive action of $\cl_{O'}$ on $\Ell_k(O')$. On the other hand, as we have just discussed, there is also a free and transitive action of $\cl_H$ on $Z_\Gamma$. 
Finally, we have the isomorphism~\eqref{eq:classgroupiso} connecting $\cl_{O'}$ to $\cl_H$, as well as a natural bijection
\[ Z_\Gamma \to \Ell_k(O') : (E, C) \mapsto \pi(E, C) := E/C. \]
It can be argued that all these maps are compatible with each other:
\begin{lemma} \label{lem:cl_O_action}
For every ideal class $[\afrak] \in \cl_{O'}$ we have
\[ [\afrak] \star \pi(E, C) = \pi([\afrak O] \star (E, C)),\]
where the left action is that of $\cl_{O'}$ on $\Ell_k(O')$, while the right action is that of $\cl_H$ on $Z_\Gamma$.
\end{lemma}
\begin{proof}
   Write $E_1 = \pi(E, C)$ and let $E_1' = \varphi_{\afrak}(E_1)$. Then $E \in \Ell_k(O)$ lies above $E_1$ via an ascending isogeny $\varphi$ with kernel $E_1[f, f\sigma]$. Likewise, there is an elliptic curve
$E' \in \Ell_k(O)$ above $E_1'$, which is the codomain of an ascending isogeny $\varphi'$ with kernel $E_1'[f, f\sigma]$. It is easy to check that we obtain a
   commuting diagram
   \begin{center}
   \begin{tikzcd}
E \arrow[r, "\varphi_{\afrak O}"] 
& E' \\
E_1 \arrow[r, "\varphi_{\afrak}" ] \arrow[u, "\varphi"]
& E_1' \arrow[u, "\varphi'"]
\end{tikzcd}
\end{center}
showing that $[\afrak O] \star E = E'$, an equality which refers to the action of $\cl_O$ on $\Ell_k(O)$; see also the proof of~\cite[Lem.\,6]{sutherland_volcanoes}. It is then immediate that $C = \ker \hat{\varphi}$ is mapped via $\varphi_{\afrak O}$ to $C' := \ker \hat{\varphi}'$, from which the statement follows. \qed
\end{proof}

\begin{remark}
From the commutativity of the above diagram it follows that
\[ \hat{\varphi}' \circ \varphi_{\afrak O} \circ \varphi = \hat{\varphi}' \circ \varphi' \circ \varphi_{\afrak} = [f] \circ \varphi_{\afrak}, \]
showing that this is the horizontal isogeny corresponding to the ideal $f\afrak$. Since the natural map
$\cl_{O'} \to \cl_O$
is not injective, one can also wonder what ideal we end up with when first choosing an isogeny $\psi : E \to E'$ and considering the horizontal isogeny 
$\hat{\varphi}' \circ \psi \circ \varphi$. One particular case is where $E = E'$, as in the case of SCALLOP. In this case the ideals corresponding to $\hat{\varphi}' \circ \varphi$ have a particularly nice interpretation: they correspond to $O'$-ideals of norm $f^2$, of the form $\afrak_{\alpha,\beta} = (f^2, f(\alpha+\beta\sigma))$, for some $\alpha+\beta\sigma \in O$ such that $N_{K/\Q}(\alpha+\beta\sigma) \not \equiv 0 \bmod f$. Indeed, there is in this case a free and transitive action of $\ker(\cl_{O'} \to \cl_O)$ on the set of all $E/C \in \Ell_k(O')$ over cyclic $f$-subgroups $C$ that are not eigenspaces of $\sigma$ acting on $E[f]$, which corresponds to the free and transitive action of $\cl_{O'}$ on $Z_{\Gamma}$ as in Theorem \ref{thm:maxlevel}. It can be proven that ideal classes $[\afrak_{\alpha,\beta}]$ in $\cl_{O'}$ are in bijection with $(\alpha:\beta) \in \mathbb{P}^1(\F_f)$ such that $N_{K/\Q}(\alpha+\beta\sigma) \not \equiv 0 \bmod f$, and that they constitute all of $\ker(\cl_{O'} \to \cl_O)$ (see also \cite[Lemma 3.2]{modular_polys_12}). Once we fix a subgroup $C = \ker \hat{\varphi}$, the action of such classes is explicitly given by: $[\afrak_{\alpha,\beta}]\star\pi(E,C) = \pi(E, (\alpha+\beta\iota(\bar\sigma))(C))]$, where $\bar \sigma$ is the complex conjugate of $\sigma$. This follows from Lemma \ref{lem:cl_O_action} and the fact that $[\afrak_{\alpha,\beta}] = [(N_{K/\Q}(\alpha+\beta\sigma), f(\alpha+\beta\bar\sigma))]$, hence $[\afrak_{\alpha,\beta} O] = [(\alpha + \beta\bar\sigma)]$.

\end{remark}

\subsection{Further examples}

\subsubsection{Cyclic torsion.}
  Assume that $O/\mfrak$ is cyclic, i.e., $b_{\mfrak} = 1$ in Lemma~\ref{lem:m-torsion}. Then we can observe
  \begin{itemize}
     \item that every $\alpha \in O$ is congruent to an integer mod $\mfrak$; in particular it can be assumed without loss of generality that $\Lambda \subseteq \Z$,
     \item every group isomorphism $\Phi : O/\mfrak \to E[\mfrak]$ is necessarily an isomorphism of $O$-modules, i.e., $Y_\Gamma = Z_\Gamma$ for any $\Gamma \subseteq \GL(O/\mfrak)$.     
  \end{itemize}
  As a more concrete example, take $\Lambda = \{1\}$ and let $f$ be a prime number such that $fO = \mfrak \cdot \overline{\mfrak}$ for some prime ideal $\mfrak \subseteq O$. By Theorem~\ref{thm:maxlevel} the ray class group $I_O(\mfrak) / P_{O, 1}(\mfrak)$ acts freely and transitively on the set $Z_{\Gamma_{O, 1}(\mfrak)}$ of $O$-oriented elliptic curves $E/k$ equipped with an eigenvector of $\sigma$ acting on $E[f]$. For instance, if $K = \Q(\sqrt{-p})$, $f$ is odd and $p \equiv 1 \bmod f$ as in CSIDH, then we can take 
 $\mfrak = (f, \sqrt{-p} - 1)$ 
  and this consists of supersingular elliptic curves over $\F_p$ with a distinguished $\F_p$-rational point of order $f$ (up to negation).


\subsubsection{Scaling by \boldmath{$n$}-th powers.} 
Let $O$ be an order and $\mfrak = (f)$ an ideal
such that $f \equiv 1 \pmod{4}$ is prime in $O$. Let $\Lambda = \Z^2 := \{\alpha^2 \mid \alpha \in \Z\}$, and assume further that $O^\times = \{\pm1\}$, such that $O^\times \subset \Lambda + f\Z$ (here we use $f \equiv 1 \pmod{4}$).
Observe that $H := P_{O,\Lambda}(\mfrak)$ is thus the set of principal ideals generated by elements that are equivalent to integers that are squares $\mod{\mfrak}$, and then it follows from \ref{prop:exact-sequence-general}, that $\#\cl_H = 2(f+1)h(O)$. As in the situation in Subsection \ref{ssec:CMaction}, we get the set $Z_{\Gamma_{O, \Lambda}(\mfrak)}$ consists of elements of the form
$(E, P, \sigma(P))$, where $P \in E[f]$, and where we identify $(E, P, \sigma(P)) \sim (E, Q, \sigma(Q))$ if and only if $P = [\lambda]Q$ for some $\lambda \in \Z$ that is a square $\bmod f$. Thus, the situation here is a fine-grained version the action of $\cl_{\Z + fO}$ on $(E, \langle P \rangle)$ from Subsection \ref{ssec:CMaction}: The slightly larger class group acts on the set that can be recognized as curves, together with one of two points of order $f$ for each subgroup of order $f$, and where the two points differ by multiplication by a non-square.

This example easily generalizes to $\Lambda = \Z^n$ for any $n \mid f-1$, such that $O^\times \subset \Lambda + f\Z$.

\subsubsection{The full class group.} 
If $\Lambda = O$, then $P_{O, \Lambda}(\mfrak) = P_O(\mfrak)$ is the group of all fractional principal ideals coprime to $\mfrak$, and we end up with the standard action of $\cl_O$ on $\Ell_k(O)$, which indeed naturally coincides with $Z_\Gamma$ where $\Gamma = \Gamma_{K, O_K}(\mfrak)$. Note that in general we do not have a well-defined action of $\cl_O$ on the larger set $Y_\Gamma$;
indeed the condition $\Lambda \subseteq O^\times \Z$ from Theorem~\ref{thm:maxlevel} is violated.
Nevertheless it makes sense to study $Y_\Gamma$ as a set; e.g., when $\mfrak = NO$ then it parametrizes $O$-oriented elliptic curves $E$ together with a basis of $E[N]$, where two bases are identified if and only if they can be transformed into one another via an endomorphism in $\iota(O)$.

\section{Security reductions and non-reductions}\label{sec:security}

Although we do not have cryptographic applications in mind, it is natural to extend the central question of~\cite{GPV} to our generalized setting: given
\[(E_1, \Phi_1),  \ (E_2, \Phi_2) \ \in \ Z_{\Gamma_{O, \Lambda}(\mfrak)},\]
how hard is it to find a generalized ideal class
$[\afrak] \in I_O(\mfrak) / P_{O, \Lambda}(\mfrak)$
such that $[\afrak](E_1, \Phi_1) = (E_2, \Phi_2)$? 
This problem is known as the vectorization problem~\cite{couveignes}, which quantum computers can solve in sub-exponential time $L_h(1/2)$, with $h$ denoting the size of the generalized class group. 

Unsurprisingly, the main conclusion from~\cite[Alg.~2]{GPV} also applies here: despite the generalized ideal class group being larger, there is an immediate reduction to the vectorization problem for $\cl_O$, potentially at the cost of a discrete logarithm computation (which may be hard classically, but succumbs to Shor's algorithm quantumly). Indeed, after finding an ideal $\afrak \in I_O(\mfrak)$ such that
$\varphi_{\afrak} : E_1 \to E_2$, 
one can find $\alpha \in O$ such that $\alpha \afrak$ moreover maps $\Phi_1$ to $\Phi_2$, via the computation of Weil pairings and discrete logarithms. 

\begin{remark}
It is worth contrasting this with actions by class groups of suborders $O' \subseteq O$. From Section~\ref{sec:orders_acting} we know that the action of $\cl_{O'}$ on $\Ell_k(O')$ is a generalized class group action in disguise. However, it would be wrong to apply the previous discussion and conclude that the corresponding vectorization problem reduces to that of $\cl_O$ acting on $\Ell_k(O)$ at the cost of discrete logarithm computations. There is again a reduction, but it proceeds by walking to $\Ell_k(O)$ via isogenies; in other words, the ``disguise" is crucial for security. An extreme case is SCALLOP~\cite{scallop,scallopHD}, where $\cl_O$ is the trivial group, but here the isogenies are of very large prime degree, hence infeasible to compute.
\end{remark}

\paragraph{Cases where vectorization becomes easier}
In view of the attacks on SIDH, the extra level structure may in fact make the vectorization problem much easier. E.g., in the case of the ray class group for scalar modulus $\mfrak = NO$, an attacker has access to a basis $P_1, Q_1$ of $E_1[N]$ along with their images under the secret isogeny $\varphi_{\afrak} : E_1 \to E_2$. Assuming we are given a bound on $\deg \varphi_{\mathfrak{a}} = N(\mathfrak{a})$ and assuming that $N$ is large enough and smooth, we can recover $\deg \varphi_{\mathfrak{a}}$ and run the algorithm from~\cite{robert2023breaking} to solve the vectorization problem in classical polynomial time. Another interesting case is the generalized class group action from Section~\ref{sec:orders_acting}, where we have a free and transitive action on oriented elliptic curves together with the kernel of a descending $f$-isogeny, see~\eqref{eq:ZGamma_desc}. Thus, here one is given access to such a kernel $C_1 \subseteq E_1[f]$ along with its image $C_2 \subseteq E_2[f]$. But then one also knows that $\sigma(C_1)$ is connected to $\sigma(C_2)$ via the same unknown scalar: the vectorization problem becomes an instance of the M-SIDH problem~\cite{msidh}, which can again be broken in overstretched cases. Moreover, if $f$ splits then we also know that the action preserves the eigenspaces of $\sigma$ acting on the $f$-torsion; as such we have access to \emph{four} subgroups together with their images, and we can reduce to the case of SIDH via~\cite[Lem.\,1]{adaptive}.


\paragraph{Supergroups of $P_O(\mfrak)$}
Finally, let us drop the overall assumption made at the beginning of Section~\ref{sec:gen_CM_action}, namely that $H \subseteq P_O(\mfrak)$: what if, conversely, our congruence subgroup $H$ is a supergroup of $P_O(\mfrak)$?
In this case the generalized class group $\cl_H = I_O(\mfrak) / H$ naturally acts on subsets
$\label{eq:orbit} 
\{ \, [\hfrak] E \, | \, \hfrak \in H \, \} \subseteq \Ell_k(O)$
of oriented elliptic curves that can be connected via an ideal in $H$. In theory, this gives a reduction from the vectorization problem for $\cl_O$ to that of the smaller group $\cl_H$: first find the class connecting $\{ \, [\hfrak] E_1 \, | \, \hfrak \in H \,  \}$ and
$\{ \, [\hfrak] E_2 \, | \, \hfrak \in H \,  \}$ and then solve a vectorization problem for $H / P_O(\mfrak)$. If this were possible, then this could be converted into a Pohlig--Hellman type reduction for class group actions. But unfortunately (or fortunately), it is unclear how to work with the sets $\{ \, [\hfrak] E \, | \, \hfrak \in H \, \}$; e.g., when merely working with a representant, one lacks tools for equality testing (which amounts to deciding whether or not two $O$-oriented elliptic curves $E_1, E_2$ are connected via an ideal in $H$).

\bibliographystyle{alpha}
\bibliography{biblio}\vspace{0.75in}

\newcommand{\etalchar}[1]{$^{#1}$}
\begin{thebibliography}{CLM{\etalchar{+}}18}

\bibitem[ACL{\etalchar{+}}24]{orientationsandcycles}
Sarah Arpin, Mingjie Chen, Kristin~E. Lauter, Renate Scheidler, Katherine~E.
  Stange, and Ha~T.N. Tran.
\newblock Orientations and cycles in supersingular isogeny graphs.
\newblock In {\em Research Directions in Number Theory: Women in Numbers V},
  pages 25--86. Springer International Publishing Cham, 2024.

\bibitem[Arp22]{arpin}
Sarah Arpin.
\newblock {\em Supersingular Elliptic Curve Isogeny Graphs}.
\newblock PhD thesis, University of Colorado Boulder, 2022.

\bibitem[BCC{\etalchar{+}}23]{secuer}
Andrea Basso, Giulio Codogni, Deirdre Connolly, Luca~De Feo, Tako~Boris
  Fouotsa, Guido~Maria Lido, Travis Morrison, Lorenz Panny, Sikhar Patranabis,
  and Benjamin Wesolowski.
\newblock Supersingular curves you can trust.
\newblock In {\em {EUROCRYPT} {(2)}}, volume 14005 of {\em Lecture Notes in
  Computer Science}, pages 405--437. Springer, 2023.

\bibitem[BF23]{msidh_artificial}
Andrea Basso and Tako~Boris Fouotsa.
\newblock New {SIDH} countermeasures for a more efficient key exchange.
\newblock In {\em {ASIACRYPT} {(8)}}, volume 14445 of {\em Lecture Notes in
  Computer Science}, pages 208--233. Springer, 2023.

\bibitem[BLS12]{modular_polys_12}
Reinier Br\"{o}ker, Kristin Lauter, and Andrew~V. Sutherland.
\newblock Modular polynomials via isogeny volcanoes.
\newblock {\em Mathematics of Computation}, 81(278):1201--1231, 2012.

\bibitem[BMP23]{festa}
Andrea Basso, Luciano Maino, and Giacomo Pope.
\newblock {FESTA}: fast encryption from supersingular torsion attacks.
\newblock In {\em Advances in cryptology---{ASIACRYPT} 2023. {P}art {VII}},
  volume 14444 of {\em Lecture Notes in Comput. Sci.}, pages 98--126. Springer,
  Singapore, [2023] \copyright2023.

\bibitem[CD20]{csurf}
Wouter Castryck and Thomas Decru.
\newblock {CSIDH} on the surface.
\newblock In {\em PQCrypto}, volume 12100 of {\em Lecture Notes in Computer
  Science}, pages 111--129. Springer, 2020.

\bibitem[CK20]{OSIDH}
Leonardo Col{\`{o}} and David Kohel.
\newblock Orienting supersingular isogeny graphs.
\newblock {\em J. Math. Cryptol.}, 14(1):414--437, 2020.

\bibitem[CL23a]{scallopHD}
Mingjie Chen and Antonin Leroux.
\newblock {SCALLOP-HD:} group action from 2-dimensional isogenies.
\newblock {\em {IACR} Cryptol. ePrint Arch.}, page 1488, 2023.

\bibitem[CL23b]{codogni}
Giulio Codogni and Guido Lido.
\newblock Spectral theory of isogeny graphs, 2023.

\bibitem[CLM{\etalchar{+}}18]{csidh}
Wouter Castryck, Tanja Lange, Chloe Martindale, Lorenz Panny, and Joost Renes.
\newblock C{SIDH}: an efficient post-quantum commutative group action.
\newblock In {\em Advances in cryptology---{ASIACRYPT} 2018. {P}art {III}},
  volume 11274 of {\em Lecture Notes in Comput. Sci.}, pages 395--427.
  Springer, Cham, 2018.

\bibitem[Coh12]{cohen2012advanced}
Henri Cohen.
\newblock {\em Advanced topics in computational number theory}, volume 193.
\newblock Springer Science \& Business Media, 2012.

\bibitem[Cou06]{couveignes}
Jean~Marc Couveignes.
\newblock Hard homogeneous spaces.
\newblock {\em {IACR} Cryptol. ePrint Arch.}, page 291, 2006.

\bibitem[Cox13]{cox}
David~A. Cox.
\newblock {\em Primes of the Form $x^2+ny^2$: Fermat, Class Field Theory, and
  Complex Multiplication}.
\newblock Pure and Applied Mathematics: A Wiley Series of Texts, Monographs and
  Tracts. Wiley, 2013.

\bibitem[CS21]{chenusmith}
Mathilde Chenu and Benjamin Smith.
\newblock {Higher-degree supersingular group actions}.
\newblock {\em Math. Cryptology}, 1(1):1--15, 2021.

\bibitem[FFK{\etalchar{+}}23]{scallop}
Luca~De Feo, Tako~Boris Fouotsa, P{\'{e}}ter Kutas, Antonin Leroux,
  Simon{-}Philipp Merz, Lorenz Panny, and Benjamin Wesolowski.
\newblock {SCALLOP:} scaling the {CSI-FiSh}.
\newblock In {\em Public-Key Cryptography - {PKC} 2023 Part {I}}, volume 13940
  of {\em Lecture Notes in Computer Science}, pages 345--375. Springer, 2023.

\bibitem[FFP24]{defeomodular}
Luca~De Feo, Tako~Boris Fouotsa, and Lorenz Panny.
\newblock Isogeny problems with level structure.
\newblock In {\em EUROCRYPT 2024 (to appear)}. Springer-Verlag, 2024.

\bibitem[FMP23]{msidh}
Tako~Boris Fouotsa, Tomoki Moriya, and Christophe Petit.
\newblock {M-SIDH} and {MDSIDH}: countering {SIDH} attacks by masking
  information.
\newblock In {\em {EUROCRYPT} {(5)}}, volume 14008 of {\em Lecture Notes in
  Computer Science}, pages 282--309. Springer, 2023.

\bibitem[FP22]{adaptive}
Tako~Boris Fouotsa and Christophe Petit.
\newblock A new adaptive attack on {SIDH}.
\newblock In {\em Topics on Cryptology -- CT-RSA}, volume 13161 of {\em Lecture
  Notes in Computer Science}, pages 322--344. Springer, 2022.

\bibitem[GPV23]{GPV}
Steven~D. Galbraith, Derek Perrin, and Jos{\'{e}}~Felipe Voloch.
\newblock {CSIDH} with level structure.
\newblock {\em {IACR} Cryptol. ePrint Arch.}, page 1726, 2023.

\bibitem[KL22]{kopplagarias2022}
Gene~S. Kopp and Jeffrey~C. Lagarias.
\newblock Class field theory for orders of number fields, 2022.

\bibitem[Len96]{lenstra_CM}
Hendrik~W.\ Lenstra.
\newblock Complex multiplication structure of elliptic curves.
\newblock {\em Journal of Number Theory}, 56(2):227--241, 1996.

\bibitem[Neu99]{neukirch}
J\"{u}rgen Neukirch.
\newblock {\em Algebraic number theory}, volume 322 of {\em Grundlehren der
  mathematischen Wissenschaften [Fundamental Principles of Mathematical
  Sciences]}.
\newblock Springer-Verlag, Berlin, 1999.
\newblock Translated from the 1992 German original and with a note by Norbert
  Schappacher, With a foreword by G. Harder.

\bibitem[Onu21]{onuki}
Hiroshi Onuki.
\newblock On oriented supersingular elliptic curves.
\newblock {\em Finite Fields Their Appl.}, 69:101777, 2021.

\bibitem[Per24]{perrin}
Derek Perrin.
\newblock Ordinary isogeny graphs with level structure.
\newblock Article in preparation, 2024.

\bibitem[Rob23a]{RobSIDH}
Damien Robert.
\newblock Breaking {SIDH} in polynomial time.
\newblock In {\em {EUROCRYPT} {(5)}}, volume 14008 of {\em Lecture Notes in
  Computer Science}, pages 472--503. Springer, 2023.

\bibitem[Rob23b]{robert2023breaking}
Damien Robert.
\newblock Breaking {SIDH} in polynomial time.
\newblock In Carmit Hazay and Martijn Stam, editors, {\em Advances in
  Cryptology - {EUROCRYPT} 2023 - 42nd Annual International Conference on the
  Theory and Applications of Cryptographic Techniques, Lyon, France, April
  23-27, 2023, Proceedings, Part {V}}, volume 14008 of {\em Lecture Notes in
  Computer Science}, pages 472--503. Springer, 2023.

\bibitem[RS06]{RostStol}
Alexander Rostovtsev and Anton Stolbunov.
\newblock Public-key cryptosystem based on isogenies.
\newblock Cryptology ePrint Archive, Paper 2006/145, 2006.
\newblock \url{https://eprint.iacr.org/2006/145}.

\bibitem[Sch87]{Schoof_NonsingPlaneCubics}
Ren\'{e} Schoof.
\newblock Nonsingular plane cubic curves over finite fields.
\newblock {\em J. Combin. Theory Ser. A}, 46(2):183--211, 1987.

\bibitem[Sil94]{SilvermanII}
Joseph~H. Silverman.
\newblock {\em Advanced topics in the arithmetic of elliptic curves}, volume
  151.
\newblock Springer, 1994.

\bibitem[Sil09]{Silverman}
Joseph~H. Silverman.
\newblock {\em The arithmetic of elliptic curves}, volume 106 of {\em Graduate
  Texts in Mathematics}.
\newblock Springer, second edition, 2009.

\bibitem[Sut13]{sutherland_volcanoes}
Andrew~V. Sutherland.
\newblock Isogeny volcanoes.
\newblock In {\em ANTS-X}, volume~1 of {\em Open Book Series}, pages 507--530.
  MSP, 2013.

\bibitem[Voi21]{Voight}
John Voight.
\newblock {\em Quaternion algebras}, volume 288 of {\em Graduate Texts in
  Mathematics}.
\newblock Springer, Cham, [2021] \copyright 2021.

\bibitem[{Wat}69]{Wat69}
William~C. {Waterhouse}.
\newblock {Abelian varieties over finite fields}.
\newblock {\em Ann. Sci. Ecole Norm. Sup.}, 2:521--560, 1969.

\bibitem[XZQ23]{xiao2023oriented}
Guanju Xiao, Zijian Zhou, and Longjiang Qu.
\newblock Oriented supersingular elliptic curves and {E}ichler orders, 2023.

\end{thebibliography}

\end{document}